\documentclass[11pt]{article}	
\usepackage[utf8]{inputenc}

\usepackage[T1]{fontenc}

\usepackage[english]{babel}
\usepackage[nottoc]{tocbibind}

\usepackage{geometry}
\usepackage{graphicx} 
\usepackage{latexsym}
\usepackage{amssymb}
\usepackage{amsfonts}
\usepackage{amsmath}
\usepackage{amsthm}
\usepackage{mathrsfs}
\usepackage{MnSymbol,wasysym}
\usepackage[shortlabels]{enumitem}
\usepackage{dsfont}
\usepackage{pst-node}

\theoremstyle{remark}

\theoremstyle{remark}
\newtheorem{example}{Example}

\usepackage{amsmath, amsthm, amssymb, thmtools}

\usepackage{amsmath, amsthm, amssymb, thmtools}

\declaretheorem[numberwithin = section, style = definition ]{theorem}

\declaretheorem[sibling = theorem]{lemma}
\declaretheorem[sibling = theorem]{proposition}
\declaretheorem[sibling = theorem]{remark}
\declaretheorem[sibling = theorem]{corollary}

\usepackage{enumitem,xparse}
\newlist{Claim}{description}{2}
\setlist[Claim]{labelindent=2em,leftmargin=*}
\newif\ifInsideClaim
\newcounter{claim}[theorem]
\newcounter{cclaim}[claim]

\let\originalqedsymbol\qedsymbol

    \newtheoremstyle{TheoremNum}
        {\topsep}{\topsep}              %%% space between body and thm
        {\itshape}                      %%% Thm body font
        {}                              %%% Indent amount (empty = no indent)
        {\bfseries}                     %%% Thm head font
        {.}                             %%% Punctuation after thm head
        { }                             %%% Space after thm head
        {\thmname{#1}\thmnote{ \bfseries #3}}%%% Thm head spec
    \theoremstyle{TheoremNum}

   \newtheoremstyle{TheoremNum}
        {\topsep}{\topsep}              %%% space between body and thm
        {\itshape}                      %%% Thm body font
        {}                              %%% Indent amount (empty = no indent)
        {\bfseries}                     %%% Thm head font
        {.}                             %%% Punctuation after thm head
        { }                             %%% Space after thm head
        {\thmname{#1}\thmnote{ \bfseries #3}}%%% Thm head spec
    \theoremstyle{TheoremNum}

\usepackage{blindtext}
\usepackage{hyperref}

\numberwithin{equation}{section} 

\newcommand{\R}{\mathbb{R}} 
\newcommand{\Z}{\mathbb{Z}}
\newcommand{\N}{\mathbb{N}}

\newcommand{\Aut}{\operatorname{Aut}}

\usepackage[affil-it]{authblk}
\usepackage{geometry,mathtools}
\newcommand{\inv}{^{-1}} 
\newcommand{\vphi}{\varphi}

\newcommand{\subeq}{\subseteq}

\newcommand{\eps}{\varepsilon}

\usepackage{ragged2e}
\usepackage{etoolbox}
\usepackage{lipsum}
\usepackage{capt-of}

\newcommand{\Sym}{\operatorname{Sym}}

\newcommand{\bigslant}[2]{{\raisebox{.2em}{$#1$}\left/\raisebox{-.2em}{$#2$}\right.}}

\usepackage{tikz}
\usepackage{graphicx}
\usepackage{mathtools,tikz-cd}

\usepackage{fancyhdr}

\newcommand\shorttitle{The stabilized automorphism group of odometers and of Toeplitz subshifts}
\newcommand\authors{Jennifer N. Jones-Baro}

\usepackage{wrapfig}
\usepackage{float}
\usetikzlibrary{quotes}
\usetikzlibrary{arrows}
\usetikzlibrary{arrows,automata,positioning}
\usepackage{ragged2e}

\usepackage{titling}
\title{\textbf{The stabilized automorphism group of odometers and of Toeplitz subshifts}}
\author{Jennifer N. Jones-Baro}
\affil{Department of Mathematics, Northwestern University}

\begin{document}

\maketitle

\vspace{-0.5cm} \begin{abstract}
We characterize the stabilized automorphism group for odometers and Toeplitz subshifts and then prove an invariance property of the stabilized automorphism group of these dynamical systems. Namely, we prove the isomorphism invariance of the primes for which the $p$-adic valuation of the period structure tends to infinity. A particular case of interest is that for torsion free odometers the stabilized automorphism group is a full isomorphism invariant.
 \end{abstract}
\section{Introduction}

Let $(X,T)$ be a dynamical system, that is, let $X$ be a compact metric space and $T$ a homeomorphism of $X$ to itself. An automorphism of $(X,T)$ is a homeomorphism $\vphi\colon X\to X$ that commutes with $T$. The set of all automorphisms of $(X,T)$ is a group under composition called the automorphism group of $(X,T)$, and we denote it by $\Aut(X,T)$.

The automorphism groups of symbolic systems have been studied since the 60's, starting with the works of Hedlund in \cite{Hedlund}. These groups continue to be studied extensively, see for example \cite{BLR,CyrKralinear, CyrKraexponential,ddmp,K&R,Ryan}. In particular, the automorphism group of Toeplitz subshifts has been studied by Donoso, Durand, Maass and Petite \cite{ddmpToe} and Salo \cite{Salo}. In this work, we study a larger group of symmetries called the \textit{stabilized automorphism group} for odometers and Toeplitz shifts.

The stabilized automorphism group was introduced in 2021 by Hartman, Kra and Schmieding \cite{stab}. Given $(X, T)$ a dynamical system, the stabilized automorphism group is the subgroup of $\operatorname{Homeo}(X)$ given by
$$\Aut^{(\infty)}(X,T) = \bigcup_{n=1}^\infty \Aut(X,T^n).$$
Building on partial results from \cite{stab}, Schmieding gave a full characterization of the stabilized automorphism group for shifts of finite type \cite{Pent}. Given natural numbers $m, n \geq 2$, the stabilized groups of the full $m$-shift and the full $n$-shift are isomorphic if and only if $m^k = n^j$ for some $k, j \in \N$.

We study the stabilized automorphism group of a class of dynamical systems with contrasting behavior to that studied by Schmieding. While mixing shifts of finite type have high complexity, we study odometer systems which have zero entropy. An important technique introduced in \cite{Pent} is the notion of \textit{local $\mathcal{P}$-entropy}, a quantity that captures the exponential growth rate of certain classes of finite subgroups in the limit that defines the stabilized automorphism group. These techniques however cannot be applied directly to our case since all odometers exhibit the same growth rate of finite groups in their stabilized automorphism group. Hence, local $\mathcal{P}$-entropy alone is not enough to distinguish two odometers by analyzing their stabilized automorphism group. However, we draw inspiration form this method to develop a new approach to the study of the growth of finite subgroups of the automorphism groups that define the stabilized automorphism group. In a similar way as to how the complexity function is a sequence that provides more information about a symbolic system than its limit, i.e. the topological entropy, by pinpointing the finite stages of the definition of the stabilized automorphism group where we see growth, we can recover the primes for which the $p$-adic valuation of the scale of the odometer tends to infinity. In particular, we show that for torsion free odometers, the stabilized automorphism group is a full isomorphism invariant. 

We use our results about odometers to study a class of subshifts called Toeplitz subshifts which have odometers as their maximal equicontinuous factor. These subshifts were first studied by Jacobs and Keane \cite{JacobsKeane} and have no restrictions in terms of their complexity. However, since they carry a lot of the same rigid structure of an odometer, we are able to use the results on odometers to conclude similar results about Toeplitz subshifts.

We defer the 
precise definitions and notation to Section \ref{sectionprelim}. In Section \ref{sectionodometers}, we study the stabilized automorphism group for odometers. The main result of this section is the following theorem. We use the notation  $\Sym(n)$ to represent the symmetric group on $n$ symbols.

\begin{theorem}\label{theoremodometer}
The stabilized automorphism group of an odometer $\Z_{(p_n)}$ with scale $(p_n)$ is isomorphic to the direct limit of a sequence
 of monomorphisms of groups of the form $\left(\Z_{(q_n)}\right)^{p_k}\rtimes \Sym(p_k)$ where $\Z_{(q_n)}$ is an odometer that is a factor of $\Z_{(p_n)}$ and $p_k$ is an element of the scale $(p_n)$.
\end{theorem} A more precise description of the stabilized automorphism group of odometers including a characterization of the monomorphisms defining the limit is given in \autoref{preciseodom}. The main technical difficulty to overcome for proving this theorem is characterizing $\Aut(\Z_{(p_n)},+\textbf{m})$ for all $m\in \Z$. We do so in \autoref{generalodom}. This characterization is different from characterizing $\Aut(\Z_{(p_n)},+\textbf{1})$ for any odometer $\Z_{(p_n)}$ as that proof relies heavily on the fact that $(\Z_{(p_n)},+\textbf{1})$ is minimal. However, when $m$ divides an element of the sequence $(p_n)$ the system $(\Z_{(p_n)},+\textbf{1})$ fails to be minimal, as we show in \autoref{generalodom}.

In Section \ref{sectiontop}, we study the stabilized automorphism group of Toeplitz subshifts and prove a similar result. 

\begin{theorem}\label{theoremtoeplitz}
Let $(X,\sigma)$ be a Toeplitz subshift with period structure $(p_n)$. Then, the stabilized automorphism group of $(X,\sigma)$ is isomorphic to the direct limit of a sequence
 of monomorphisms of groups of the form $\Aut(T,\tau)^{p_k}\rtimes \Sym(p_k)$ where $(T,\tau)$ is a Toeplitz shift and $p_k$ is an element of the sequence $(p_n)$.
\end{theorem}

A more precise description of the stabilized automorphism group of Toeplitz subshifts including a characterization of the monomorphisms defining the limit is given in \autoref{precisetop}. Similarly to the theorem about odometers, the main technical difficulty is characterizing $\Aut(X,\sigma^m)$ for all $m\in \Z$. We do so in \autoref{topgeneralcase}.

As an immediate corollary to the previous theorems, since amenable groups are preserved under direct limits we have the following.
\begin{corollary}\label{amenable}
Both the stabilized automorphism group of an odometer and the stabilized 
automorphism of a Toeplitz subshift are ameanable.
\end{corollary}

Odometers are completely classified by an equivalence relation on their scale (see \cite{DownSur}). Let $(p_n)$ be the scale of an odometer. For each prime number $p$, denote by $\nu_p(n)$ the p-adic valuation of the integer $n$, i.e.
$\nu_p(n)=\max\{k\geq 0: p^k \text{ divides } n\}.$ For each prime the multiplicity function at $p$ of the scale $(p_n)$ is given by
$\operatorname{\bf{v}}_p(p_n)=\lim_{n\to \infty}\nu_p(p_n).$ Two scales $(p_n)$ and $(s_n)$ are equivalent if and only if $ \operatorname{\bf{v}}_p(p_n)=\operatorname{\bf{v}}_p(s_n) \text{ for all primes } p$. Two odometers are isomorphic if and only if their scales are equivalent. In Section \ref{sectioninva}, we study the finite subgroups at each level of the sequences in Theorems \ref{theoremodometer} and \ref{theoremtoeplitz} to prove the isomorphism invariance of the primes for which the the multiplicity function at $p$ is infinite. We use this to derive our main invariance results.

\begin{theorem}
 \label{finalcor1}
Let $(\Z_{(p_n)},+\textbf{1})$ and $(\Z_{(q_n)},+\textbf{1})$ be torsion free odometers with scales $(p_n)$ and $(q_n)$ respectively. If $\Aut^{(\infty)}(\Z_{(p_n)},+\textbf{1})$ and $\Aut^{(\infty)}(\Z_{(q_n)},+\textbf{1})$ are isomorphic as groups then $\Z_{(p_n)}$ and $\Z_{(q_n)}$ are isomorphic as groups.
\end{theorem}

\begin{theorem}
 \label{finalcor2}
Let $(X,\sigma)$ and $(T,\tau)$ be torsion-free Toeplitz subshifts with scales $(p_n)$ and $(q_n)$ respectively. If $\Aut^{(\infty)}(X,\sigma)$ and $\Aut^{(\infty)}(T,\tau)$ are isomorphic as groups, then $(p_n)$ is equivalent to $(q_n)$.
\end{theorem}

\subsection{Acknowledgments}
This material is based upon work supported by the National Science Foundation Graduate Research Fellowship under Grant No. DGE-1842165. The author is also grateful for the helpful discussions and feedback received throughout this project from Bryna Kra and Scott Schmieding. They would also like to thank the referee for their careful and insightful review of this paper, and for the comments, corrections and suggestions they made that greatly improved this work. Additionally, they would like to thank 
Kaitlyn Loyd, Nick Lohr, Nir Avni, Stephan Snegirov, Bastián Espinoza and Adam Holeman for their helpful comments.

\section{Preliminaries}\label{sectionprelim}
\subsection{Background}

A \textit{topological dynamical system} (or simply a \textit{system}) is a pair $(X,T)$ where $X$ is a compact metric space with metric $\operatorname{d}\colon X\times X\to \R$ and $T\colon X\to X$ is a homeomorphism. In the particular case when $X$ is a compact topological group and $T$ acts by group translation by a fixed element in $X$, we call the dynamical system $(X,T)$ a \textit{group rotation}. The \textit{orbit of a point} $x\in X$ is denoted by $\mathcal{O}_T(x)=\{T^n(x): n\in \Z\}$. Given a subset $U\subseteq X$ we define $\mathcal{O}_T(U)=\bigcup\limits_{x\in U}\mathcal{O}_T(x)$. A system is minimal if the orbit of every point $x\in X$ is dense in $X$. A subset $U\subseteq X$ is called a \textit{minimal component} of $(X,T)$ if $U$ is closed, $T$-invariant and the restriction of $T$ to $U$ makes $(U,T|_U)$ a minimal system.

Given two topological dynamical systems $(X,T)$, $(Y,S)$ a continuous surjection $\pi\colon X\to Y$ such that $\pi\circ T=S\circ \pi$ is called a \textit{factor map}. If such map exists, we say $(Y,S)$ is a factor of $(X,T)$. If in addition $\pi$ is a bijection, we say $(X,T)$ and $(Y,S)$ are \textit{conjugate} systems.

Let $G$ and $H$ be two topological groups. We say that $G$ and $H$ are \textit{isomorphic as topological groups} if there exists a group isomorphism $\phi \colon G \to H$ that is also a homeomorphism. Not all group isomorphism are necessarily topological isomorphism and to avoid confusion, we refer to a usual group isomorphism as an \textit{algebraic isomorphism} and denote it with the symbol $\cong$. Moreover, two group rotations $(G,g)$ and $(H,h)$ are conjugate if and only if there exists a topological isomorphism $\phi \colon G \to H$ such that $\phi(g)=\phi(h)$.

We say the system $(X,T)$ is \textit{equicontinuous} if for any $\eps>0$ there exists $\delta>0$ such that if $\operatorname{d}(x,y)\leq\delta$ for $x,y\in X$ then for any $n\in \Z$ we have $\operatorname{d}(T^n(x),T^n(y))\leq \eps$. Every minimal equicontinuous system is conjugate to a group rotation (see for example \cite{kurka}).

\subsection{Automorphism group and stabilized automorphism group}

An \textit{automorphism} of a system $(X,T)$ is a homeomorphism $\vphi$ of $X$ such that $\vphi\circ T=T\circ \vphi$. The set of all automorphisms of $X$ forms a group under composition which we denote by $\Aut(X,T)$ and call the \textit{automorphism group} of $(X,T)$. A commonly used result in the literature is the following lemma. We include the proof for completeness.

\begin{lemma}\label{minimalcomponents}
Let $(X,T)$ be a dynamical system and $\vphi\in\Aut(X,T)$. Then, $U\subseteq X$ is a minimal component of $(X,T)$ if and only if $\vphi (U)$ is a minimal component.
\end{lemma}
\begin{proof}
Let $U\subseteq X$ be a minimal component of $(X,T)$ and let $\vphi\in\Aut(X,T)$. Since $\vphi$ is a homeomorphism of $X$, $\vphi(U)$ is closed. Additionally, since $U$ is $T$-invariant we have that $T(\vphi(U))=\vphi(T(U))\subseteq\vphi(U)$. Hence, $\vphi(U)$ is $T$-invariant. Take $y\in \vphi(U)$. Since $\vphi$ is a bijection, there exists $x\in U$ with $\vphi(x)=y$ and since $U$ is a minimal component $\overline{\mathcal{O}}_T(x)=U$. Because $\vphi$ is an automorphism of $(X,T)$, we have that $\overline{\mathcal{O}}_T(y)=\overline{\mathcal{O}}_T(\vphi(x))=\vphi(U)$. We conclude $\vphi(U)$ is a minimal component. To show that $\vphi\inv(U)$ is a minimal component we repeat the proof with $\vphi\inv$ instead of $\vphi$.
\end{proof}

As introduced by Hartman, Kra and Shmieding in \cite{stab}, for $(X,T)$ a a dynamical system, we define the \textit{stabilized automorphism group} of $(X,T)$ to be the subgroup of $\operatorname{Homeo}(X)$ given by 
$$\Aut^{(\infty)}(X,T)=\bigcup_{n=1}^\infty\Aut(X,T^n).$$

\begin{remark}\label{tree}
It is obvious that if $i$ divides   $j$ then $\Aut(X,T^i)\subseteq \Aut(X,T^j)$. The stabilized automorphism group is equivalently defined as the direct limit (colimit in the categorical sense) of the following diagram where the arrows represent inclusions.
\end{remark}
\begin{center}
    \includegraphics[scale=0.15]{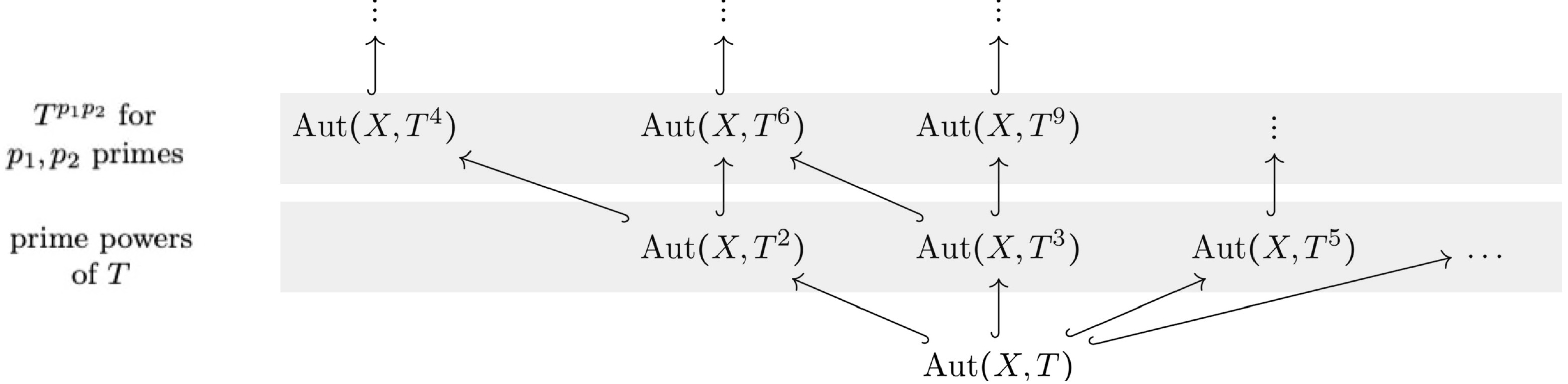}
    \captionof{figure}{}
\end{center}
% \begin{center}
% \begin{tikzcd}
%  \vdots                       & \vdots                             & \vdots                                                  &                         &       \\
%  {\Aut(X,T^4)} \arrow[u, hook]      & {\Aut(X,T^6)} \arrow[u, hook]            & {\Aut(X,T^9)} \arrow[u, hook]                                 & \vdots                  &       \\
%  \text{all prime powers of T} & {\Aut(X,T^2)} \arrow[lu, hook] \arrow[u, hook] & {\Aut(X,T^3)} \arrow[lu, hook] \arrow[u, hook]                      & {\Aut(X,T^5)} \arrow[u, hook] & \dots \\
%                               &                                    & {\Aut(X,T)} \arrow[lu, hook] \arrow[u, hook] \arrow[ru, hook] \arrow[rru, hook] &                         &      
%  \end{tikzcd}
%  \captionof{figure}{}
% \end{center}

\begin{proposition}
\label{semidirectproduct}
Let $(X,T)$ be a minimal dynamical system. Assume that for $k>1$ we have that $(X,T^k)$ has $n>1$ minimal components $U_1, U_2,..., U_n$ such that $X=\bigcup_{i=1}^n U_i$. If the dynamical systems $(U_i, T^k|_{U_i})$ are conjugate for $i=1,2,...,n$ then there exists an algebraic group isomorphism $$\chi\colon\Aut(X,T^k)\to [\Aut(U_1, T^k|_{U_1})]^n\rtimes \Sym(n) $$ where $\Sym(n)$ is the symmetric group on $n$ symbols, satisfying for all $\vphi,\phi\in \Aut(X,T^k)$ with $\chi(\vphi)=((a_1,a_2,...,a_n),\pi_1)$ and $\chi(\phi)=((b_1,b_2,...,b_n),\pi_2) $
\begin{align}
\begin{split}
    \chi(\vphi\circ\phi)&=((a_1,a_2,...,a_n),\pi_1)\cdot((b_1,b_2,...,b_n),\pi_2) 
    \\&=(\pi_2\inv(a_1,a_2,...,a_n)(b_1,b_2,...,b_n),\pi_1\circ\pi_2).\label{multiplication}
    \end{split}
\end{align}
where $\pi_2\inv(a_1,a_2,...,a_n)=(a_{\pi_2\inv(1)},a_{\pi_2\inv(2)},...,a_{\pi_2\inv(n)})$ and $\pi_1\circ\pi_2$ denotes the composition of the composition of functions (as opposed to cycle concatenation).
\end{proposition}

We point out that the isomorphism $\chi$ is not canonical. It requires making a choice of isomorphism between $\Aut(U_i, T^k|_{U_i})$ and $\Aut(U_1, T^k|_{U_1})$ for all $i=1,...,n$.

\begin{proof}
The minimal components of $(X,T^k)$ form a partition of $X$ into closed sets. Since $(U_i, T^k|_{U_i})$ are conjugate for $i=1,2,...,n$ define $\vphi_{i,i+1}$ for $i=1,2,...,n-1$ to be a conjugacy between $(U_i, T^k|_{U_i})$ and $(U_{i+1}, T^k|_{U_{i+1}})$, and define $\vphi_{n,1}=\vphi_{1,2}\inv\circ\vphi_{2,3}\inv\dots\circ\vphi_{n-2,n-1}\inv\circ\vphi_{n-1,n}\inv$. Hence, we have the following commutative diagram:

\begin{center}
\begin{tikzcd}\label{minimaldia}
U_1 \arrow[r, "{\varphi_{1,2}}"'] \arrow[d, "T^k"] & U_2 \arrow[r, "{\varphi_{2,3}}"'] \arrow[d, "T^k"] & \dots \arrow[r, "{\varphi_{n-1,n}}"'] & U_n \arrow[d, "T^k"] \arrow[lll, "{\varphi_{n,1}}", bend right] \\
U_1 \arrow[r, "{\varphi_{1,2}}"]                   & U_2 \arrow[r, "{\varphi_{2,3}}"]                   & \dots \arrow[r, "{\varphi_{n-1,n}}"]  & U_n. \arrow[lll, "{\varphi_{n,1}}", bend left]                  
\end{tikzcd}

\captionof{figure}{}
\end{center} Additionally, define for $i,j\in\{1,2,...,n\}$ with $i\leq j$ define $\vphi_{i,j}=\vphi_{j-1,j}\circ\vphi_{j-2,j-1}\circ...\circ\vphi_{i+1,i+2}\circ\vphi_{i,i+1}$ and $\vphi_{j,i}=\vphi_{ij}\inv$.

By \autoref{minimalcomponents} an automorphism of $(X, T^k)$ defines a permutation on the set of minimal components of $(X, T^k)$. So we can define a map $\rho\colon  \Aut (X, T^k) \to \Sym(n)$ by sending each automorphism to its corresponding permutation on the set of minimal components. 

Let $\pi\in \Sym(n)$. Define 
$\Phi_\pi$ such that $U_i$ is mapped to $U_{\pi(i)}$ via $\vphi_{i,\pi(i)}$. Since minimal components are closed and disjoint $\Phi_\pi$ is continuous and since it commutes with $T^k$ on each minimal component we can conclude that $\Phi_\pi$ is an automorphism of $(X,T)$. Notice $\rho(\Phi_\pi)=\pi$. Thus $\rho$ is surjective.

We can construct an automorphism $\Psi$ of $(X,T)$ by choosing a particular automorphism $f_i$ of each minimal component $(U_i, T^k|U_{i})$ and defining $\Psi\equiv f_i$ on $U_i$. That is, $\Psi$ does not permute the minimal components and only acts on each one by their specified automorphism. Since minimal components are closed and disjoint $\Psi$ is continuous and since it commutes with $T^k$ on each minimal component we can conclude that $\Psi$ is an automorphism of $(X,T^k)$. Hence, we can define a group monomorphism $\iota\colon \Aut(U_1, T^k|_{U_1})^n\to \Aut(X,T^k)$, since $\Aut(U_i, T^k|_{U_i})$ is isomorphic to $\Aut(U_1, T^k|_{U_1})$ for $i=1,...,n$. Notice $\Psi\in {\Aut(U_1, T^k|_{U_1})}^n $ and $\rho(\iota(\Psi))= e$ where $e$ is the identity in $\Sym(n)$. Therefore, we have the following short exact sequence:\begin{center}
\begin{tikzcd}
1 \arrow[r] & {\Aut(U_1, T^k|_{U_1})}^n \arrow[r, "\iota"] & {\operatorname{Aut}(X,T^k)} \arrow[r, "\rho"] & \Sym(n){} \arrow[r] \arrow[l, dotted, bend right=49] & 1
\end{tikzcd}
\end{center}
Using the fact that the diagram in Figure 2 commutes, we can define a splitting of this sequence as the map from $\Sym(n)$ to $\operatorname{Aut}(X, T^k)$ by sending each permutation $\pi\in\Sym(n)$ to $\Phi_\pi$ as defined above. Hence,  $$\Aut(X,T^k)\cong [\Aut(U_1,T^k|_{U_1})]^n \rtimes \Sym(n).$$

The formula for the multiplication in \ref{multiplication} follows immediately.
\end{proof}

\subsection{Odometers}
We give a brief review of odometers, for a more complete exposition see \cite{DownSur}. 

Let $(p_n)$ be a sequence of natural numbers such that $p_n$ divides $p_{n+1}$. We call any such sequence a \textit{scale}. We define the \textit{odometer with scale $(p_n)$} as the subgroup of $\prod_{n=1}^{\infty} \bigslant{\Z}  {p_n\Z}$ given by
$$\Z_{(p_n)}=\{(x_n)\in \prod_{n=1}^{\infty} \bigslant{\Z}  {p_n\Z} : \,\,x_n\equiv x_{n+1}\,\, \operatorname{mod } \,p_n, \, \text{ for all } n\in\N\}.$$
The odometer $\Z_{(p_n)}$ can also be defined as the inverse limit $\lim_{\xleftarrow[]{}n} \bigslant{\Z}  {p_n\Z}$ of the canonical homomorphisms $\bigslant{\Z}  {p_{n+1}\Z} \to \bigslant{\Z}{p_{n}\Z}$. The natural dynamics on an odometer $\Z_{(p_n)}$ is given by the addition by $\operatorname{\text{\bf{1}}}=(1,1,1,...)$. It is not difficult to see that it is a minimal equicontinuous topological dynamical system called an \textit{odometer} and denote by $(Z_{(p_n)},+\operatorname{\text{\bf{1}}})$. We call both the group $\Z_{(p_n)}$ and the system $(\Z_{(p_m)},+\textbf{1})$ an odometer, which one we are referring to is clear from the context. In particular, the subgroup $\langle\operatorname{\text{\bf{1}}}\rangle$ is dense in $\Z_{(p_n)}$ and is isomorphic to $\Z$. We denote the multiples of $\textbf{1}$ by $\textbf{m}=m\operatorname{\text{\bf{1}}}=(m \mod{p_1},m\mod{p_2},m\mod{p_3},...)$ for all $m\in\N$.

For each prime number $p$, denote by $\nu_p(n)$ the \textit{p-adic valuation} of the integer $n$, i.e.
$$\nu_p(n)=\max\{k\geq 0: p^k \text{ divides } n\}.$$ Given an odometer $\Z_{(p_n)}$ the sequence $(\nu_p(p_n))_{n\geq1}$ is non-decreasing and we can define for each prime the \textit{multiplicity function} at $p$ as
$$\operatorname{\bf{v}}_p(p_n)=\lim_{n\to \infty}\nu_p(p_n).$$

We can endow an odometer $\Z_{(p_n)}$ with the metric 
$$\operatorname{d}(x,y)=2^{-\inf\{i\in \N\,\,:\,\, x_i-y_i\neq 0\}}$$
for any $x=(x_n)$ and $y=(y_n)\in \Z_{(p_n)}$. With this metric $\Z_{(p_n)}$ is a compact topological group. 

The question of when two odometers are isomorphic (as topological groups or simply algebraically) is completely understood by the following theorem. 

\begin{theorem}[see for example \cite{DownSur}]  \label{odometerisomorphism}
Two odometers $\Z_{(p_n)}$ and $\Z_{(s_n)}$ are isomorphic both algebraically and as topological groups if and only if $\operatorname{\bf{v}}_q(p_n)=\operatorname{\bf{v}}_q(s_n)$ for all primes $q$. Moreover, for an odometer $\Z_{(p_n)}$  we have the following group isomorphism 
\begin{equation}\label{odomdecomp}
    \Z_{(p_n)}\cong\left(\prod_{p\in I}\Z_{(p^n)}\right)\times\left( \prod_{p \in F}\Z/{p^{\textbf{v}_p(p_n)}}\Z\right),
\end{equation}
where $I=\{ p \text{ prime}:\operatorname{\bf{v}}_p(p_n)=\infty\}$ and $F=\{ p \text{ prime}:1<\operatorname{\bf{v}}_p(p_n)<\infty\}$. The image of $\operatorname{\text{\bf{1}}}$ under this isomorphism is  $((\operatorname{\text{\bf{1}}}, \operatorname{\text{\bf{1}}},...),(1,1,1,...)).$ 
\end{theorem}

An immediate consequence of the previous theorem is that the torsion subgroup of an odometer can be written explicitly as
$$T(\Z_{(p_n)})=\prod_{p \in F}\Z/{p^{\textbf{v}_p(p_n)}}\Z,$$
where $F=\{ p \text{ prime}:1<\operatorname{\bf{v}}_p(p_n)<\infty\}$.

This theorem leads us to define the following equivalence relation on scales. Two scales $(p_n)$ and $(s_n)$ are equivalent, denoted by $(p_n)\sim(s_n)$, if and only if $ \operatorname{\bf{v}}_p(p_n)=\operatorname{\bf{v}}_p(s_n) \text{ for all primes } p$. It is easy to check that this is an equivalence relation. Two odometers $\Z_{(p_n)}$ and $\Z_{(s_n)}$ are isomorphic if and only if  $(p_n)\sim(s_n)$.

As stated in \cite{DownSur}, an odometer $\Z_{(p_n)}$ is a factor of another odometer $\Z_{(q_n)}$ if and only if for all $k\in \N$, $p_k$ divides $q_\ell$ for some $\ell \in \N$. This allows us to define the partial ordering $(p_n)\preccurlyeq (s_n)$ if and only if all the following hold:
\begin{enumerate}[(1)]
    \item For all primes $p$, $\operatorname{\bf{v}}_p(p_n)=\infty$ if and only if $\operatorname{\bf{v}}_p(s_n)=\infty$.
 \item For all primes $p$ such that $\operatorname{\bf{v}}_p(s_n)<\infty$ we have that $\operatorname{\bf{v}}_p(p_n)\leq\operatorname{\bf{v}}_p(s_n)$.
\end{enumerate}

\begin{remark}

By \autoref{odometerisomorphism}, two scales $(p_n)$ and $(s_n)$ define isomorphic odometers if and only if $(p_n)\sim(s_n)$. That is, an odometer is completely determined by the sequence $(\operatorname{{\bf{v}}_q(p_n))}_{q\text{ a prime}}\in (\N\cup\{\infty\})^\infty$. Additionally, if $(p_n)\preccurlyeq (q_n)$ then the odometer $\Z_{(p_n)}$ is a factor of the odometer $\Z_{q_n}$.
\end{remark}

We say a scale $(p_n)$ is a \textit{prime scale} if $p_{n+1}/p_{n}$ is prime for all $n\in \N$. Notice that for any scale $(p_n)$ there exists a prime scale $(\Tilde{p}_n)$ such that $(p_n)\sim (\Tilde{p}_n)$.

We say an odometer $\Z_{(p_n)}$ or equivalently a scale $(p_n)$ is
\begin{enumerate}[(i)]
    \item \textit{finite} if there exits $N\in \N$ such that $p_m=p_N$ for all $m\geq N$;
    \item \textit{torsion free} if $\operatorname{\bf{v}}_p(p_n)\in \{0,\infty \}$ for all primes $p$.
\end{enumerate}
For a more detailed classification of odometers see \cite{DownSur}.

\begin{remark}
From now on, we assume any scale $(p_n)$ is not finite as otherwise the group $\Z_{(p_n)}$ is finite and the dynamical system $(\Z_{(p_n)},+\textbf{1})$ is periodic.
\end{remark}

Odometers classify all equicontinuous dynamical systems on a totally disconnected infinite space.

\begin{theorem}[see for example \cite{kurka}] \label{odomuniversality}
Let $(X,T)$ be a minimal equicontinuous dynamical system on a totally disconnected infinite space $X$. Then $(X,F)$ is conjugate to an odometer. 
\end{theorem}

The automorphism groups of odometers are completely classified.

\begin{proposition}[see for example \cite{ddmpToe}]\label{isomautodom}
Let $\Z_{(p_n)}$ be an odometer, then  $\Aut(\Z_{(p_n)},+\operatorname{\text{\bf{1}}})\cong \Z_{(p_n)}$ as groups.
\end{proposition}

This theorem establishes the full isomorphism invariance of the automorphism group for odometers. 

\begin{corollary}
If $\Z_{(p_n)}$ and $\Z_{(s_n)}$ are two odometers, then $\Z_{(p_n)}\cong\Z_{(s_n)}$ if and only if $\Aut(\Z_{(p_n)},+\textbf{1})\cong\Aut(\Z_{(s_n)},+\textbf{1})$.
\end{corollary}

\subsection{Symbolic systems}
Let $\mathcal{A}$ be a finite set. We define $\mathcal{A}^\Z$ to be the set of bi-infinite sequences $(x_i)_{i\in\Z}$ with $x_i\in \mathcal{A}$ for all $i\in \Z$.  When endowed with the metric 
$$\operatorname{d}((x_i),(y_i))=2^{-\inf\{|i|:x_i\neq y_i\}},$$
$\mathcal{A}^\Z$ is a compact metric space. We define the \textit{left shift} $\sigma\colon\mathcal{A}^\Z\to\mathcal{A}^\Z$ by $(\sigma x)_i = x_{i+1}$ for all $i\in \Z$. If $X \subseteq \mathcal{A}^\Z$ is closed and $\sigma$-invariant, then the dynamical system $(X,\sigma|_X)$ is called a \textit{subshift}. We omit the notation $\sigma|_X$ and just denote a subshift by $(X,\sigma)$. 

For $w = (w_1,... , w_{n})\in \mathcal{A}^n$, we define \textit{the cylinder set} as $$[w]=\{x\in \mathcal{A}^\Z:x_i =w_i \text{ for all } 0\leq i\leq n\}.$$
The collection of cylinder sets $\{\sigma^i([w]) : w \in \mathcal{A}^\ast, i \in\Z\}$ where $ \mathcal{A}^\ast=\bigcup\limits_{j=1}^{\infty}\mathcal{A}^j$ is a basis for the topology of $\mathcal{A}^\Z$.

The \textit{language} of a subshift $(X, \sigma)$ is $$\mathcal{L}(X):=\{w\in\mathcal{A}^\ast: [w
]\cap X\neq \emptyset\}$$
and any $w \in \mathcal{L}(X)$ is called a \textit{word} in the language. For all $n\in\N$, define $\mathcal{L}_n(X)$ to be set of words of length $n$ in $\mathcal{L}(X)$.  The \textit{complexity of a subshift} is $P_X\colon \N\to\N$ defined as $P_X(n)=\# \mathcal{L}_n(X)$.

\subsection{Toeplitz subshifts}\label{toepsec}
A sequence $u=\{u_t\}_{t\in \Z}$ is a \textit{Toeplitz sequence} if for all 
$ n\in \Z \text{ there exists }  m\in \N \text{ such that for all } k\in \Z\text{ we have } u_{n}=u_{n+km}.$ For any $p\in \N$, define $$\operatorname{per}_{p}(u)=\{k\in \N\,|\,  u_k=u_{k+pm} \text{ for all }m\in\Z\}.$$ Then $u$ is a Toeplitz sequence if there exists a sequence of integers $(p_n)$ such that $p_n$ divides $  p_{n+1}$ for all $n\in \N$ and 
$$\bigcup_{n\in \N}\operatorname{per}_{p_n}(u)=\Z.$$
We call the sequence $(p_n)$ a \textit{scale} of $u$. Similarly to odometers, we say a scale $(p_n)$ is a \textit{prime scale} if $p_{n+1}/p_{n}$ is prime for all $n\in \N$.

We say that $p_n$ is an \textit{essential period} of $u$ if for any $1 \leq p < p_n$ the sets $\operatorname{per}_p(u)$ and $\operatorname{per}_{p_n}(u)$ do not coincide. If the sequence ${p_n}$ is formed by essential periods we call it a \textit{period structure} of $u$.

If $u$ is a Toeplitz sequence we define the \textit{Toeplitz subshift given by u} to be $(X_u,\sigma_u)$ where $X_u=\overline{\mathcal{O}_{\sigma}(u)}$ and $\sigma_u=\sigma|_{X_u}$. We omit the sub-index to simplify the notation and denote by $(X, \sigma)$ the respective Toeplitz subshift. Toeplitz subshifts were defined by Jacobs and Keane who also showed that every Toeplitz shift is minimal $\cite{JaK}$.

Let $(X,\sigma)$ be a Toeplitz subshift given by the Toeplitz sequence $u$. From now on, we assume $u$ is not periodic as otherwise the system $(X,\sigma)$ is periodic. An element $x\in X$ is called a \textit{Toeplitz orbital}. It is important to note that a Toeplitz orbital may not be a Toeplitz sequence as some of its coordinates may not be periodic. Since $u$ is not a periodic sequence, points in $X$ that are not Toeplitz sequences necessarily exist (compare to Cor. 4.2 in \cite{baake}). If $x$ is a Toeplitz sequence in $X$ we call it a \textit{regular point}. We denote by $R$ the set of all regular points in $X$. The singleton fibers of the map $\pi\colon X\to \Z_{(p_n)}$ from $X$ to its maximal equicontinuous factor $(\Z_{(p_n)},+1)$ correspond to the regular points in $X$ and form a dense $G_\delta$ subset of $X$ (see for example \cite{DownSur}). It is clear that any period that occurs in $x$ is also a period that occurs in $u$. We define the \textit{periodic part of x} as
 $$\operatorname{P(x)}=\bigcup_{n\in \N}\operatorname{per}_{p_n}(x),$$
and the \textit{aperiodic part of $x$} as
$$\operatorname{A}(x)=\Z\backslash\operatorname{P}(x).$$
We call the \textit{$p$-skeleton} of $x=(x_i)\in X$ the part of $x$ which is periodic with period $p$. To make this precise, we define the $p$-skeleton to be the sequence obtained from $x$ by replacing $x_i$ by a new symbol "?" for all $i\nin \operatorname{per}_{p}(x)$.

Regarding the aperiodic part, we have the following useful properties. 
\begin{lemma}[see for example \cite{DownSur}]\label{aper} Let $(X,\sigma)$ be a Toeplitz subshift and $x\in X$
\begin{enumerate}[(a)]
    \item For any $n\in \operatorname{A}(x)$ there is no $l>0$ such that $x_{n+kl}=x_{n}$ for all $k\in \Z$.
\item Every finite pattern occurring along the aperiodic part of x also occurs along some periodic part.
\end{enumerate}

\end{lemma}

The following key lemma about Toeplitz subshifts was proved by Williams.
\begin{lemma}[Williams \cite{Williams}]\label{skeletonlemma}
Let $(X,\sigma)$ be a Toeplitz subshift given by the Toeplitz word $u$ with period structure $(p_n)$. For each $i\in \N$, $n\in \Z/p_i\Z$ define $A_n^i=\{\sigma^m(u):m\equiv n \mod p_i\}$. Then \begin{enumerate}[(i)]
    \item $\{\overline{A_n^i}:n\in\Z/p_i\Z\}$ is a partition of $X=\overline{\mathcal{O}}_\sigma(u)$ into relatively open (and closed) sets,
    \item $\overline{A_m^j}\subseteq \overline{A_n^i}$ for $i<j$ and $m\equiv n \mod p_i$,
    \item $\sigma(\overline{A_n^i})=\overline{A_{n+1}^i}$.
\end{enumerate}
\end{lemma}

Toeplitz subshifts have been fully characterized up to topological conjugacy by the following theorem.
\begin{theorem}[see for example \cite{DownSur}]\label{caracToep}
    A dynamical system $(X,\sigma)$ is conjugate to a Toeplitz subshift if and only if it satisfies the following three properties
    \begin{enumerate}[(i)]
    \item $(X,T)$ is minimal,
    \item $(X,T)$ is an almost one-to-one extension of an odometer,
    \item $(X,T)$ is symbolic.
\end{enumerate}
\end{theorem}

\begin{remark} [see for example \cite{Williams}]\label{11extension}
    The map that gives rise to property (ii) of the previous lemma is constructed as follows. Let $(X,\sigma)$ be a Toeplitz subshift with period structure $(p_n)$. For $g=(x_i)\in \Z_{(p_n)}$ we set 
    $$A_g=\bigcap_{i=0}^\infty \overline{A_{x_i}^i}.$$

    We define the factor map $\pi\colon (X,\sigma)\to (\Z_{(p_n)},+\textbf{1})$ by $\pi\inv(g)=A_g$. Then $\pi(y)=\pi(y')$ for $y,y'\in X$ if and only if $y$ and $y'$ have the same $p_i$-skeleton for all $i\in\N$. In particular, $\pi$ is one-to-one on the set of Toeplitz sequences in $X$.
\end{remark}

As a consequence of property (ii) of the previous theorem, if $(X, \sigma)$ is a the Toeplitz subshift given by the Toeplitz sequence $u$ with period structure $(p_n)$ then $(\Z_{(p_n)}, +\operatorname{\text{\bf{1}}})$ is its maximal equicontinuous factor (see for example \cite{Williams}). Another consequence of this is the following result.

\begin{lemma}[see for example \cite{ddmp}]
The automorphism group of a Toeplitz subshift is isomorphic to a subgroup of its corresponding odometer maximal equicontinuous factor.
\end{lemma}

\begin{remark}
As a consequence of the previous lemma, the automorphism group of a Toeplitz subshift is abelian.
\end{remark}

We use some similar terminology for Toeplitz subshifts as for odometers. We say a Toeplitz subshift given by the Toeplitz word $u$ with period structure $(p_n)$ is \textit{torsion free} if its corresponding odometer maximal equicontinuous factor is torsion free.

\section{The stabilized automorphism group of an odometer}\label{sectionodometers}
This section is dedicated to characterizing the stabilized automorphism group of odometers. In order to study the stabilized automorphism group of odometers, we first analyze $\Aut(\Z_{(p_n)}, +\text{\textbf{m}})$ for all $m\in \N$. We start by proving the following proposition.

\begin{proposition}\label{generalodom}
Let $\Z_{(p_n)}$ be an odometer with scale $(p_n)$ and set $m\in \N$. Let $d\geq 0$ be such that for some $k_0\in \N$ we have that $(p_k,m)=d$ for $k\geq k_0$ and $k_0$ is the smallest integer with this property. Then $(\Z_{(p_n)},+\textbf{m})$ has $d$ minimal components each of them conjugate to the odometer with scale $(\frac{p_n}{d})_{n\geq k_0}$. Furthermore, $\Aut(\Z_{(p_n)},+\textbf{m})\cong \Z_{(p_{n}/d)_{n\geq k_0}}^d \rtimes \Sym(d)$ and is isomorphic to a subgroup of $\Aut(\Z_{(p_n)},+\textbf{p}_{k_0})$.
\end{proposition}

\begin{proof}
We will first assume $d=1$. We know $(\Z_{(p_n)},+\textbf{m})$ is minimal by Lemma 2.1 in \cite{ddmpToe}. Since $(\Z_{(p_n)}, +\text{\bf{m}})$ is a minimal equicontinuous dynamical system on a totally disconnected space by \autoref{isomautodom} and \autoref{odomuniversality} $(\Z_{(p_n)}, +\text{\bf{m}})$ is conjugate to the odometer $(\Z_{(p_n)}, +\text{\bf{1}})$ and $\Aut(\Z_{(p_n)}, +\text{\bf{m}})\cong \Z_{(p_n)}$. However, since many groups have subgroups isomorphic to themselves, including some odometers, this is not enough to conclude $\Aut(\Z_{(p_n)},+\textbf{m})=(\Z_{(p_n)},+\textbf{1})$. We show this next.

It is obvious that $\Aut(\Z_{(p_n)}, +\text{\bf{1}})\subseteq \Aut(\Z_{(p_n)}, +\text{\bf{m}})$, we are left with proving the other inclusion. Take $\vphi\in \Aut(\Z_{(p_n)}, +\text{\bf{m}})$ and $\eps>0$.  Since $\vphi$ is continuous, by our definition of the metric in $\Z_{(p_n)}$ there exists $N\in\N$ such that for all $(x_i),(y_i)\in\Z_{(p_n)}$ if $x_j=y_j$ for all $j\leq N$ then $\operatorname{d}(\vphi(x_i),\vphi(y_i))<\eps/2$. Pick $M\in \N$ such that for all $(x_i),(y_i)\in\Z_{(p_n)}$ if $x_j=y_j$ for all $j\leq M$ then $\operatorname{d}((x_i),(y_i))<\eps/2$. Define $K=\max\{N,M\}$. By Bézout's identity, since $(p_K,m)=1$ there exist $a,b\in\N$ such that 
$$am=bp_K+1.$$ Because $+\textbf{m}$ commutes with $\phi$ and by our choice of $K$ we have that for all $x=(x_i)\in\Z_{(p_n)}$ \begin{align*}
    \operatorname{d}(\vphi(x+\textbf{1}),\vphi (x)+\textbf{1})&\leq\operatorname{d}(\vphi(x+\textbf{1}),\vphi(x+a\textbf{m}))+\operatorname{d}(\vphi(x+a\textbf{m}),\vphi (x)+\textbf{1})
    \\&=\operatorname{d}(\vphi(x+\textbf{1}),\vphi(x+a\textbf{m}))+\operatorname{d}(\vphi(x)+a\textbf{m},\vphi(x)+\textbf{1})
    \\&\leq \eps/2+\eps/2 =\eps.
\end{align*}
Where the last inequality follows from the fact that $x+a\textbf{m}$ and $x+a\textbf{1}$ agree on the first $K$ coordinates. We conclude $\vphi(x+\textbf{1})=\vphi x+\textbf{1}$, hence $\vphi\in \Aut(\Z_{(p_n)},+\textbf{1})$. This proves $\Aut(\Z_{(p_n)},+\textbf{m})=\Aut(\Z_{(p_n)},+\textbf{1})\cong \Z{(p_n)}$.

Assume now that $d>1$. By \autoref{odometerisomorphism}, $(\Z_{(p_n)}, +\textbf{1})$ is conjugate to an odometer $(\Z_{(p'_n)}, +\textbf{1})$ with period structure $(p'_n)$ such that $p'_1=d$ and $\operatorname{\textbf{v}}_q(p_n)=\operatorname{\textbf{v}}_q(p'_n)$. Without loss of generality, we assume $p_1=d$. Since the first coordinate of elements in $\Z_{(p_n)}$ belongs to $\Z/d\Z$ the addition $+\textbf{m}$ fixes the first coordinate. For $j=0,1,...,d-1$, we define the subsets of $\Z_{(p_n)}$ $$U_j=\{(x_i)\in \Z_{(p_n)}\,\,:\,\,x_1=j\}.$$ Notice that these are clopen sets invariant under the action $+\textbf{m}$. Define the map $\vphi\colon {U_j}\to \Z_{(p_{n+1}/d)_{n\in \N}}$ by $$\vphi((x_i)_{i\in \N})=(\frac{x_{i+1}-j}{d})_{i\in \N}.$$
Then $\vphi$ is a homeomorphism and the following diagram commutes
\begin{center}
\begin{tikzcd}
 U_j \arrow[d, "\vphi"'] \arrow[r, "+\textbf{m}"] & U_j \arrow[d, "\vphi"] \\
\mathbb{Z}_{(p_{n+1}/d)} \arrow[r, "+\textbf{m/d}"']                         & \mathbb{Z}_{(p_{n+1}/d)}.                     
\end{tikzcd}
\end{center} This implies the action of $+\textbf{m}$ restricted to $U_j$ is conjugate to $(\Z_{(p_{n+1}/d)}, +\textbf{s})$, where $s=m/d$. By the case $d=1$, $U_j$ is a minimal component. Hence, the number of minimal components of  $(\Z_{(p_n)}, +\text{\bf{m}})$ is $d$. Moreover, we have that each minimal component is conjugate to $(\Z_{(p_{n+1}/d)}, +\text{\bf{1}})$ and we have the identity $\Aut(U_j,+\textbf{m})=\Aut(U_j,+\textbf{s}).$

 By the case $d=1$, we have that the automorphism group of each minimal component under the action $+\text{\bf{s}}$ is isomorphic to $\Z_{(p_{n+1})/d}$.
 Moreover, as consequence of \autoref{semidirectproduct}, we have 
$$\Aut(\Z_{(p_n)}, +\text{\bf{m}})=\Aut(\Z_{(p_n)}, +\text{\bf{d}})\cong \Z_{(p_{n+1}/d)}^d \rtimes \Sym(d).$$

Given $(q_n)$ any equivalent period structure, we have shown the inclusion
$$\Aut(\Z_{(q_n)}, +\text{\bf{m}})=\Aut(\Z_{(q_n)}, +\text{\bf{d}})\subseteq \Aut(\Z_{(q_n)}, +\text{\bf{{q'}}}_{k_0}),$$
where $k_0\in \N$ is such that $(q_k,m)=d$ for $k\geq k_0$.
\end{proof}

\begin{remark}\label{BVtranslation}
    One can translate the previous proof to one relying on the Bratelli-Vershik representation of odometers. To do this, for $k\in \N$ and $0\leq i<p_k$ consider the sets 
    $$U_{k,i}=\{(x_n)\in \Z_{(p_n)}\,:\,x_k=i \}.$$
    The sets $\{U_{k,0},U_{k,1},...,U_{k,p_k-1}\}$ correspond to the floors of the $k$-the Kakutani-Rokhlin partition of the odometer. The action $+\textbf{m}$ on the collection of sets $\{U_{k,0},U_{k,1},...,U_{k,p_k-1}\}$ works like addition by $m$ in $\Z/p_k\Z$ by identifying $U_{k,i}$ with $i\in \Z/p_k\Z$. Thus, $U_{k,i}+\textbf{m}=U_{k,i+m \mod  p_k}$ and $U_{k,i}+r\textbf{m}=U_{k,i}$ if and only if $rm\in p_k\Z$. Take $r_k$ the smallest number such that $U_{k,i}+r\textbf{m}=U_{k,i}$. For a large enough $k$, a minimal component of $(\Z_{p_n},+\textbf{m})$ is a union of elements in $\{U_{k,0},U_{k,1},...,U_{k,p_k-1}\}$ that form a single orbit under the action $+\textbf{m}$. In the language of Bratelli-Vershik diagrams, a minimal component of $(\Z_{p_n},+\textbf{m})$ is the induced system given by the $r$-paths that correspond to the sets $U_{k,i}$ of the level $k$ which are in the same orbit under the action $+\textbf{m}$. This becomes more apparent after the proof of \autoref{topgeneralcase} using the map $\pi$ in \autoref{11extension}.
\end{remark}

\begin{corollary}\label{corautodom}
The stabilized automorphism group of an odometer is $$\Aut^{(\infty)}(\Z_{(p_n)},+\text{\bf{1}})=\bigcup_{n=1}^\infty \Aut(\Z_{(p_n)},+\text{\bf{p$_n$}})$$
where the union is taken inside $\operatorname{Homeo}(\Z_{(p_n)})$. Additionally, this statement is true for all scales equivalent to $(p_n)$.
\end{corollary}

\begin{proof}
By \autoref{generalodom}, we have that $\Aut(\Z_{(p_n)},+\textbf{m})\subseteq\bigcup\limits_{n=1}^\infty \Aut(\Z_{(p_n)},+\text{\bf{p$_n$}})$ for all $m\in \N$.
Moreover, by \autoref{odometerisomorphism} this is true for all scales equivalent to $(p_n)$.
\end{proof}

The last ingredient we need before proving our characterization of the stabilized automorphism group of odometers is the following algebraic lemma. This is a basic fact about direct limits, for a proof see for example Proposition 10.3 in \cite{Lang}.

\begin{lemma} \label{algebra}
Let $\{G_i\}_{i\in \N}$ and $\{H_i\}_{i\in \N}$ be groups and $f_{i}\colon G_i\to G_{i+1}$, $k_{i}\colon H_i\to H_{i+1}$ group homomorphisms for all $i\in \N$. Define $\hat{G}$ to be the direct limit $\lim_{\to}G_i$ and $\hat{H}$ to be the direct limit $\lim_{\to}H_i$. If there exist group isomorphisms $\vphi_i\colon G_i\to H_i$, for all $i\in \N$ such that the following diagram commutes 
\begin{center}
\begin{tikzcd}
G_i \arrow[r, "f_i"] \arrow[d, "\varphi_i"] & G_{i+1} \arrow[d, "\varphi_{i+1}"] \\
H_{i} \arrow[r, "k_i"]                    & H_{i+1}                     
\end{tikzcd}
\end{center}then $\hat{G}$ and $\hat{H}$ are isomorphic as groups.
\end{lemma}

\begin{theorem}\label{preciseodom}

The stabilized automorphism group of an odometer $\Z_{(p_n)}$ with scale $(p_n)$ is isomorphic to the direct limit of the following sequence
\begin{center}
    \begin{tikzcd}
\Z_{(p_n)} \arrow[r, "j_0"]                                     & \Z_{(p_{n+1})/p_1}^{p_1} \rtimes \Sym({p_1}) \arrow[r, "j_1"]               & \Z_{(p_{n+2})/p_2}^{p_2} \rtimes \Sym({p_2}) \arrow[r, "j_2"]               & \Z_{(p_{n+3})/p_3}^{p_3} \rtimes \Sym({p_3}) \arrow[r, "j_3"]               & \dots
\end{tikzcd}
\end{center}
 where $j_k$  are injective maps.
\end{theorem}

The injective maps $j_k$ from the previous theorem are constructed explicitly as follows. Let $\vphi\colon \Aut(\Z_{(p_n)},+\text{\bf{$p_k$}})\to \Z_{(p_{n+k})}^{p_k} \rtimes \Sym(p_k)$, for all $k\in\N\cup\{0\}$ be the isomorphisms described in \autoref{generalodom}, define $j_k=\vphi_{k+1}\circ i_k\circ \vphi_k\inv$, where $i_k\colon\Aut(X,\sigma^k)\mapsto\Aut(X,\sigma^{k+1})$ is the natural inclusion.

 \begin{proof}
By \autoref{corautodom}, the stabilized automorphism group of $\Z_{p_n}$ is $\Aut^{(\infty)}(\Z_{(p_n)},+\text{\bf{1}})=\bigcup_{n=1}^\infty \Aut(\Z_{(p_n)},+\text{\bf{p$_n$}})$
where the union is taken inside $\operatorname{Homeo}(\Z_{(p_n)})$. This is equivalent to taking the direct limit of the following diagram. 
\begin{center}
\begin{tikzcd}
{\Aut (\Z_{(p_n)},+\text{\bf{1}})} \arrow[r,"i_0", hook] & {\Aut (\Z_{p_{n}},+\text{\bf{p$_1$}})} \arrow[r,"i_1", hook] & {\Aut (\Z_{p_{n}},+\text{\bf{p$_2$}})} \arrow[r, "i_2", hook]  & {\Aut (\Z_{p_{n}},+\text{\bf{p$_3$}})} \arrow[r, "i_3", hook] & \dots
\end{tikzcd}
\end{center}

So, we have the following commutative diagram \begin{center}
  \begin{tikzcd}
{\Aut (\Z_{(p_n)},+\text{\bf{1}})} \arrow[r,"i_0", hook] \arrow[d, "\vphi_0"] & {\Aut (\Z_{p_{n}},+\text{\bf{p$_1$}})} \arrow[r,"i_1", hook] \arrow[d, "\vphi_1"] & {\Aut (\Z_{p_{n}},+\text{\bf{p$_2$}})} \arrow[r, "i_2", hook] \arrow[d, "\vphi_2"] & {\Aut (\Z_{p_{n}},+\text{\bf{p$_3$}})} \arrow[r, "i_3", hook] \arrow[d, "\vphi_3"] & \dots \\
\Z_{(p_n)} \arrow[r, "j_0"]                                     & \Z_{(p_{n+1})}^{p_1} \rtimes \Sym(p_1) \arrow[r, "j_1"]               & \Z_{(p_{n+2})}^{p_2} \rtimes \Sym(p_2) \arrow[r, "j_2"]               & \Z_{(p_{n+3})}^{p_3} \rtimes \Sym(p_3) \arrow[r, "j_3"]               & \dots
\end{tikzcd}  
  \end{center}
where the direct limit of the top row defines the stabilized automorphism group of $\Z_{(p_n)}$. By \autoref{algebra}, we conclude that this direct limit is equal to the direct limit of the bottom row which is what we wanted to prove.
 \end{proof}
 
As a direct corollary of \autoref{preciseodom}, since amenability is preserved under direct limits we conclude \autoref{amenable} for the case of odometers.

\section{The stabilized automorphism group of a Toeplitz subshift} \label{sectiontop}
This section is devoted to the proof of \autoref{theoremtoeplitz}. We begin our study of the stabilized automorphism group of Toeplitz subshifts by proving the following proposition.

\begin{proposition}\label{topgeneralcase}
Let $(X,\sigma)$ be a Toeplitz subshift with period structure $(p_n)$ and set $m\in \N$. Let $d> 0$ be such that for some $k_0\in \N$ we have that $(p_k,m)=d$ for $k\geq k_0$ and $k_0$ is the smallest integer with this property. Then, there exists a Toeplitz subshift $(T,\tau)$ with period structure $(\frac{p_n}{d})_{n\geq k_0}$ such that $(X,\sigma^m)$ has $d$ minimal components each of them conjugate to $(T,\tau)$. Furthermore, $\Aut(X,\sigma^m)\cong \Aut(T,\tau)^d\rtimes \Sym(d)$ and is isomorphic to a subgroup of $\Aut(\Z_{(p_n/d)_{n\geq k_0}},+\textbf{d})$.
\end{proposition}

\begin{proof}

Define $A_j^i$ as in \autoref{skeletonlemma}. By property (iii) of this lemma, we have that $\sigma^m$ permutes the elements in $\{\overline {A_0^{i}},\overline {A_{1}^i},...,\overline {A_{p_{i-1}}^{i}}\}$ as $\sigma^m(\overline{A_j^{i}})=\overline{A_{j+m\mod p_{i}}^{i}}$.  Hence, for each $i\in \N$, the smallest integer $r_i$ such that $\sigma^{r_im}(\overline{A_j^{i}})=\overline{A_j^{i}}$ is $r_i$ such that 
\begin{align*}r_im=\operatorname{lcm}\{m,p_i\}=\frac{mp_i}{(m,p_i)}. 
\end{align*}

Since $k_0$ is the smallest integer such that we have that $(p_k,m)=d$ for $k\geq k_0$, by property (ii) in \autoref{skeletonlemma}, we have that for any $i,j$ the orbit of $\overline{A_j^i}$ under $\sigma^m$ can be expressed as the union of $r_{k_0}$ elements in $\{\overline {A_0^{k_0}},\overline {A_{1}^{k_0}},...,\overline {A_{p_{{k_0}-1}}^{k_0}}\}$. 

Define $U_i=\mathcal{O}_{\sigma^m}(\overline{A_i^{k_0}})$ for $i=1,...,d$. (Notice that $d=(p_{k_0},m)=p_{k_0}/r_{k_0}$.) Since $(X,\sigma)$ is minimal, using property (ii) of \autoref{skeletonlemma} we can show that every orbit in $(U_i,\sigma|_{U_i})$ is dense for $i=1,...,d$, i.e $(U_i,\sigma|_{U_i})$ is minimal. Hence, $(X,\sigma)$ has $d$ minimal components. In particular, if $(m,p_n)=1$ for all $n\in \N$ then $(X,\sigma^m)$ is minimal.

By \autoref{generalodom}, since $(X,\sigma)$ is an almost one-to-one extension of $(\Z_{(p_n)},+\textbf{1})$ and since $(\Z_{(p_n)},+\textbf{1})$ has exactly $d$ minimal components we conclude that every $U_i$ is the inverse image under the almost one-to-one extension map from $X$ to $\Z_{(p_n)}$ of a minimal component of $(\Z_{(p_n)},+\textbf{m})$. Since every minimal component on $(\Z_{(p_n)},+\textbf{m})$ is conjugate to the odometer with scale $(\frac{p_n}{d})_{n\geq k_0}$ we have that $(U_i,\sigma|_{U_i})$ is an almost one-to-one extension of the odometer $\Z_{\frac{p_n}{d})_{n\geq k_0}}$. Since $(X,\sigma)$ is a symbolic system, so is $(X,\sigma^m)$. Hence, since $\sigma$ is a conjugacy between the minimal components of $(X,\sigma^m)$, by \autoref{caracToep}, we conclude that there exists a Toeplitz subshift $(T,\tau)$ with period structure $({\frac{p_n}{d})_{n\geq k_0}}$ such that $(U_i,\sigma|_{U_i})$ is is conjugate to $(T,\tau)$ for $i=1,...,d$. By \autoref{semidirectproduct}, we conclude 
$$\Aut(X,\sigma^m)\cong \Aut(T,\tau)^d\rtimes \Sym(d).$$

\end{proof}

\begin{remark}
    As stated in \autoref{BVtranslation}, one can use the map $\pi$ in \autoref{11extension} to construct an explicit representation of the minimal components of $(\Z_{p_n},+\textbf{m})$ without modifying the period structure by taking $\pi\inv(U_i)$ for $i=1,...,d$.
\end{remark}

So far, we have shown that  if $(X,\sigma)$ be a Toeplitz subshift with period structure $(p_n)$ and $(m,p_n)=1$ for all $n\in \N$, then 
$\Aut(X,\sigma)\cong\Aut(X,\sigma^m).$ We will turn this statement into an equality in the following proposition.

\begin{proposition}
    Let $(X,\sigma)$ be a Toeplitz subshift with period structure $(p_n)$. If $(m,p_n)=1$ for all $n\in \N$, then 
$$\Aut(X,\sigma)=\Aut(X,\sigma^m).$$
\end{proposition}

\begin{proof}
     We know $\Aut(X,\sigma)\subseteq\Aut(X,\sigma^m).$ We now prove the other inclusion. Let $\vphi\in\Aut(X,\sigma^m)$. We must show $\phi\circ\sigma(x)=\sigma\circ\vphi(x)$, for all $x\in X$. Let $R$ be the set of regular points in $X$ defined as in \autoref{toepsec}. Since $R$ is a dense $G_\delta$ subset of $X$ and $\vphi$ is a homeomorphism $\vphi\inv(R)\cap R$ is a dense $G_\delta$ set by Baire's Category Theorem. Hence it is enough to prove this statement for $x\in \vphi\inv(R)\cap R$ by continuity of $\vphi$ and $\sigma$.

Let $x\in \vphi\inv(R)\cap R$. Notice $x$  and $\vphi(x)$ are both Toeplitz sequences. We show $|\vphi\circ\sigma(x)-\sigma\circ\vphi(x)|=0$. Let $\eps_i$ be a decreasing sequence of positive numbers such that $\eps_i\to 0$. For each $\eps_i$ define $M_i\in\N$ to be an integer such that for any two elements $z,y\in X$ if $z_i=y_i$ for all $|i|\leq M$ then $|z-y|\leq \eps_i.$ Let $p_s$ be the largest period of the coordinates $x_j$ with $|j|\leq M_i+1$ of $x$. Let $p_\ell$ be the largest period of the coordinates $\vphi(x)_j$ with $|j|\leq M_i+1$ of $\vphi(x)$. Notice that $p_s$ divides $ p_\ell$
or $p_\ell$ divides $ p_s$, so fix $\hat p_i$ the larger of the two. Since $(m,\hat p_i)=1$ there exist integers $a_i, b_i$ such that $a_i=1+b\hat p_i$. Then we have that 
\begin{align}
    &\sigma^{a_im}(x)\to \sigma(x),\qquad\text{ as $i\to \infty$}\label{con1}\\&\sigma^{a_im}(\vphi(x))\to \sigma(\vphi(x)),\qquad\text{ as $i\to \infty$},\label{con2}
\end{align}
by our construction of the $a_i$'s. Notice we have the following inequality 
\begin{align*}
    d(\vphi\circ\sigma(x),\sigma\circ\vphi(x))&\leq  d
    (\sigma\circ\vphi(x),\sigma^{a_im}(\vphi(x)))+d(\sigma^{a_im}(\vphi(x)),\vphi\circ\sigma(x))\\&= d(\sigma(\vphi(x)),\sigma^{a_im}(\vphi(x)))+d(\vphi(\sigma^{a_im}(x)),\vphi(\sigma(x))).
\end{align*}
By equations \ref{con1} and \ref{con2} and since $\vphi$ is continuous, the right hand side goes to $0$ as $i\to\infty$. Thus $|\vphi\circ\sigma(x)-\sigma\circ\vphi(x)|=0$. Since $\vphi\inv(R)\cap R$ is a $G_\delta$ subset of $X$ and $\vphi$ is continuous, we can conclude that $\vphi\circ\sigma(x)=\sigma\circ\vphi(x)$, for all $x\in X$. Hence, $\vphi\in\Aut(X,\sigma)$.

\end{proof}

\begin{theorem}\label{precisetop}
Let $(X, \sigma)$ be a Toeplitz subshift with period structure $(p_n)$. Then, the stabilized
automorphism group of $(X, \sigma)$ is the direct limit of the sequence

\begin{center}
    \begin{tikzcd}
{\Aut (X,\sigma)} \arrow[r,hook] & {\Aut (X,\sigma^{p_1})} \arrow[r,hook] & {\Aut (X,\sigma^{p_2})} \arrow[r,hook] & {\Aut (X,\sigma^{p_3})} \arrow[r,hook] & \dots
\end{tikzcd}
\end{center}where the maps are the natural inclusion of each automorphism group into the next. \end{theorem}

Observe that in \autoref{topgeneralcase} we described $\Aut(X,\sigma^{p_n})$ for all $n$.

\begin{proof}
Let $m\in \N$. If $(m,p_n)=1$ for all $n\in \N$ by part (i) of \autoref{topgeneralcase} $\Aut(X,\sigma^m)=\Aut(X,\sigma)$. Hence, $\Aut(X,\sigma^m)\subseteq \bigcup\limits_{n=1}^\infty \Aut(X,\sigma^{p_n})$. If $(m,p_n)\neq 1$ take 
$M=\lim_{k\to \infty}\operatorname{lcm}(m,p_k)$. By \autoref{topgeneralcase} $\Aut(X,\sigma^{m})=\Aut(X,\sigma^{M})$. By our construction of $M$, there exists $k\in \N$ such that $M$ divides $ p_k$. Hence, $\Aut(X,\sigma^{m})=\Aut(X,\sigma^{M})\subseteq\Aut(X,\sigma^{p_k})$. This implies, $\Aut(X,\sigma^m)\subseteq \bigcup\limits_{n=1}^\infty \Aut(X,\sigma^{p_n})$.
\end{proof}

As a direct corollary of \autoref{precisetop}, since amenability is preserved under direct limits and Toeplitz subshifts have abelian automorphism groups we conclude \autoref{amenable} for the case of Toeplitz subshifts.

\section{Invariance of the stabilized automorphism group for odometers and Toeplitz subshifts up to scale equivalence} \label{sectioninva}
This section is dediacated to proving \autoref{inva}.

\subsection{Invariance}

\begin{lemma}\label{lemmaforinva}
Let $(\Z_{(p_n)},+\textbf{1})$ be an odometer with scale ${(p_n)}$ and $q$ be a prime such that $\operatorname{\textbf{v}}_q(p_n)=\infty$. If $\textbf{x}\in\Z_{(p_n)}$ is an element of infinite order, then there exists $\lambda \in \Aut^\infty(\Z_{(p_n)},+\textbf{1})$ such that, for some $k\in \N$, $\lambda$ commutes with $+q^k\textbf{x}$ but not with $+\textbf{x}$.
\end{lemma}
\begin{proof}
Let $\textbf{x}=(x_i) \in\Z_{(p_n)}$ be an element of infinite order and let $N\in \N$ be the first integer such that $q$ divides $p_N$ and $x_N\neq0$. We can always find such an integer since $\textbf{x}$ is not a torsion element and $\textbf{v}_q(p_n)=\infty$. Set $k=\nu_q(x_N)+1$.

Define $p'_n=p_{n+N-1}$ then $\Z_{(p_n)}\cong\Z_{(p'_n)}$ via the isomorphism $(y_i)\mapsto (y_{i+N-1})$. Take $\textbf{x}'=(x_{i+N-1})$. By \autoref{generalodom}, $(\Z_{p'_n},+\textbf{q}^k)$ has $q^k$ minimal components, denote them by $V_1, V_2,...,V_{q^k}$, each a union of sets of the form 
$U_i=\{(y_i)\in \Z_{(p'_n)}:x_1=j\}$
for $j=0,1,...,p_N-1$ and $(U_j,+\textbf{q}|_{U_j})$ is conjugate to an odometer. Let $\vphi$ be a non-trivial element in $\Aut(U_j,+\textbf{q}|_{U_j})$. Let $\lambda$ be the image under the map $\iota$ described in \autoref{semidirectproduct} of the map that acts via $\vphi$ on $V_1$ and the identity on all other minimal components.

Notice that $U_j+\textbf{x}=U_{j+x'_1 \mod p'_1}$. Since $x'_1\nequiv 0  \mod p'_1$ and by our choice of $k$, $+\textbf{x}$ permutes the minimal components $V_j$ in a non trivial permutation corresponding to an element of the subgroup isomorphic to $\Z/q^k\Z$ of the group of permutations of the $V_j$'s identified with $\Sym(q^k).$ On the other hand, $+q^k\textbf{x}$ leaves all of the minimal components $V_j$ invariant. Since odometers are abelian, one can easily see that $+q^k\textbf{x}$ commutes with $\lambda$ but $+\textbf{x}$ does not. \end{proof}
% \begin{proof}
% Let $\textbf{x}=(x_i) \in\Z_{(p_n)}$ be an element of infinite order and let $m\in \N$ be the first integer such that $x_m\neq0$ and $q$ divides $p_m$. We can always find such an integer since $\textbf{x}$ is not a torsion element and $\textbf{v}_q(p_n)=\infty$. Pick $\ell\in \N$ the first integer such that $\ell\equiv x_i \mod p_i$ for $0\leq i\leq m$. We can do this since the orbit of $+\textbf{1}$ is dense in $\Z_{(p_n)}$. Then $q\cdot\boldsymbol{\ell}$ agrees with $q\cdot \textbf{x}$ on the first $m$ coordinates. Moreover, $+q\cdot \textbf{x}$ fixes the minimal components of $+{q\ell}\textbf{1}$ but $+\textbf{x}$ does not. That is, $+q\cdot \textbf{x}$ commutes with the a subgroup of $\Aut(\Z_{(p_n)},+\boldsymbol{q\ell})$ isomorphic to $\{\textbf{0}\}\times\Sym(k)$ for $k$ the number of minimal components of $+\boldsymbol{q\ell}$, consisting of all permutations of the minimal components of $+\boldsymbol{q\ell}$. Notice that $\textbf{x}$ does not commute with all elements of this subgroup meaning there exist $\lambda \in\Aut(\Z_{(p_n)},+\boldsymbol{q\ell})$ that commutes with $+q\textbf{x}$ but not $+\textbf{x}$. This concludes the proof.
% \end{proof}

\begin{corollary}\label{corolary2}
Let $(X,\sigma)$ be a Toeplitz subshift with scale ${(p_n)}$ and let $q$ be a prime such that $\textbf{v}_q(p_n)=\infty$. If $x\in\Aut(X,\sigma)$ is an element of infinite order, then there exists $\lambda \in \Aut^\infty(X,\sigma)$ such that $\lambda$ commutes with $x^q$ but not with $x$.
\end{corollary}
\begin{proof}
This proof is identical to the last proof since $\Aut(X,\sigma)$ is isomorphic to a subgroup of an odometer and in the proof of \autoref{lemmaforinva} we only used the existence of a non trivial element in the automorphism group of the minimal components of $(X,T^{k})$ for all $k\in \N$. 
\end{proof}

\begin{theorem}\label{inva}
Let $(\Z_{(p_n)},+\textbf{1})$ and $(\Z_{(q_n)},+\textbf{1})$ be two odometers with scales $(p_n)$ and $(q_n)$ respectively and let $s$ be a prime. If $\textbf{v}_s(p_n)=\infty$ and $\Aut^{(\infty)}(\Z_{(p_n)},+\textbf{1})\cong\Aut^{(\infty)}(\Z_{q_n},+\textbf{1})$, then $\textbf{v}_s(p_n)=\infty$.
\end{theorem}

\begin{proof}
Proceeding by contradiction, assume $\vphi\colon \Aut^{(\infty)}(\Z_{(p_n)},+\textbf{1}) \to \Aut^{(\infty)}(\Z_{(q_n)},+\textbf{1})$
is a group isomorphism and $\ell=\textbf{v}_s(p_n)<\infty$. Take $j=s^\ell$ and $\gamma=\vphi(+\textbf{1})$. We can assume there exists $k>0$ such that $\gamma^j\in \Aut(\Z_{(q_n)},+\textbf{q}_k)\cong \Z_{(q_{n+k}/q_k)}^{q_k}\rtimes \Sym(q_k)$. With some abuse of notation, we assume $\gamma^j \in \Z_{(q_{n+k}/q_k)}^{q_k}\rtimes \Sym(q_k)$ as opposed to taking the image of $\gamma^j$ under the appropriate isomorphism. We also use $\Aut(\Z_{(q_n)},+\textbf{q}_k)$ and $\Z_{(q_{n+k}/q_k)}^{q_k}\rtimes \Sym(q_k)$ interchangeably (as they are isomorphic) according to the best interpretation required for our reasoning. Define $\pi\colon \Z_{(q_{n+k}/q_k)}^{q_k}\rtimes \Sym(q_k)\to \Sym(q_k)$ to be the canonical projection. 

Let $y\in\N$ be such that $\pi(\gamma^{j\cdot y})=e$. Write $\gamma^{j\cdot y}=((\gamma_1,\gamma_2,...,\gamma_{q_k}),e)$. Since $+\textbf{1}$ is an infinite order element and $\vphi$ is an isomorphism, so is $\gamma$. Hence, there exists $\gamma_i\in \Z_{(q_{n+k}/q_k)}$ such that $\gamma_i$ is an infinite order element in $\Z_{(q_{n+k}/q_k)}$.

Let us restrict our attention to the action of $\gamma_i$ on the $i$-th minimal component of $+\textbf{q}_k$. Since $\textbf{v}_s(p_n)=\infty$, by \autoref{lemmaforinva} there exists an element in $\Aut^\infty(\Z_{(q_n)},+\textbf{1})$ that commutes with $\gamma_i^{y\cdot s}$ but not with $\gamma^y$. By \autoref{semidirectproduct} \autoref{multiplication}, $\gamma^{j\cdot y\cdot s}=((\gamma_1^s,\gamma_2^s,...,\gamma_{q_k}^s),e)$. We conclude, there exists $\lambda$ that commutes with $\gamma^{j\cdot y\cdot s}$ but not with $\gamma^{j\cdot y}$. This is a contradiction to \autoref{generalodom} because it implies $\Aut(\Z_{(p_n)},+{\textbf{j}\cdot \textbf{y}\cdot \textbf{s}})\neq\Aut(\Z_{(p_n)},+{\textbf{j}\cdot \textbf{y}})$. We conclude $\textbf{v}_s(p_n)=\infty$. \end{proof}

\begin{corollary}\label{coro1}
Let $(X,\sigma)$ and $(Y,\tau)$ be two Toeplitz subshifts with scales $(p_n)$ and $(q_n)$ respectively let $s$ be a prime. If $\textbf{v}_s(p_n)=\infty$ and $\Aut^{(\infty)}(X,\sigma)\cong\Aut^{(\infty)}(Y,\tau)$, then $\textbf{v}_s(p_n)=\infty$.
\end{corollary}

\begin{proof}
Proceeding by contradiction, assume $\vphi\colon \Aut^{(\infty)}(X,\sigma) \to \Aut^{(\infty)}(Y,\tau)$
is a group isomorphism and $j=\textbf{v}_s(p_n)<\infty$. Take $\gamma=\vphi(\sigma)$. We can assume there exists $k>0$ such that $\gamma^j\in \Aut(Y,\tau^{q_k})\cong \Aut(\hat Y,\hat \tau )^{q_k}\rtimes \Sym(q_k)$ where $(\hat Y,\hat S )$ is a Toeplitz subshift with scale $(q_{n+k}/q_k)$. The rest of the proceeds identically to that of \autoref{inva}, using \autoref{corolary2} to reach a contradiction to \autoref{theoremtoeplitz}. \end{proof}

\begin{corollary}\label{coro2}
Let $(\Z_{(p_n)},+\textbf{1})$ and $(\Z_{(q_n)},+\textbf{1})$ be torsion free odometers with scales $(p_n)$ and $(q_n)$ respectively. If $\Aut^{(\infty)}(\Z_{(p_n)},+\textbf{1})$ and $\Aut^{(\infty)}(\Z_{(q_n)},+\textbf{1})$ are isomorphic as groups, then $(p_n)\sim (q_n)$ and $\Z_{(p_n)}\cong\Z_{(q_n)}$. 
\end{corollary}

We have proved \autoref{finalcor1} and \autoref{finalcor2} by proving \autoref{coro1} and \autoref{coro2}.

\subsection{Limitations}
For the case of torsion free odometers, we have established a full automorphism invariance of the stabilized automorphism group. However, in the case of Toeplitz subshifts we do not get such a strong result. To illustrate this, we present the following example of two Toeplitz sequences that admit $(2^n)$ as a scale (not an essential period structure for the second example) that are not conjugate but our methods fail to identify them as different systems. These examples can be found in \cite{DownSur}.

\begin{example}\label{ex1}
Consider the Toeplitz sequence $u=(u_i)$ with symbols $0$ and $1$ constructed in the following iterative process. First consider the sequence $x^0=(x_i)$ where every entry $x_i=?$, where $?$ indicates a place-holder for the entries that have not yet been determined. We define $x^{j}$ as follows, if $j$ is odd, we fill every second available position in $x^{j-1}$ with $0$; if $j$ is even, we fill every second available position in $x^{j-1}$ with $1$. We define $u$ to be the limit of this process. The following chart depicts the construction of each $x^j$.
\[ 
\begin{matrix*}[r]
x^0=&...&?&?&?&?&?&?&?&?&?&?&?&?&?&?&?&?&?&?&?&?&?&?&?&?&?&?&?&...\\x^1=&...&?&0&?&0&?&0&?&0&?&0&?&0&?&0&?&0&?&0&?&0&?&0&?&0&?&0&?&...\\
x^2=&...&1&0&?&0&1&0&?&0&1&0&?&0&1&0&?&0&1&0&?&0&1&0&?&0&1&0&?&...
\\x^3=&...&1&0&0&0&1&0&?&0&1&0&0&0&1&0&?&0&1&0&0&0&1&0&?&0&1&0&0&...
\\x^4=&...&1&0&0&0&1&0&1&0&1&0&0&0&1&0&?&0&1&0&0&0&1&0&1&0&1&0&0&...
\\\vdots
\\u=&...&1&0&0&0&1&0&1&0&1&0&0&0&1&0&0&0&1&0&0&0&1&0&1&0&1&0&0&...
\end{matrix*} 
\]
\end{example}

\begin{example}\label{ex2}
For this example, consider the Toeplitz sequence $w=(u_i)$ with symbols $0$ and $1$ constructed in a similar iterative process. First consider the sequence $y^0=(y_i)$ where every entry $y_i=?$. To define $y^{j}$ as we fill every second available position in $y^{j-1}$ by alternating between $0$ and $1$. We define $w$ to be the limit of this process. The following chart depicts the construction of each $y^j$.
\[ 
\begin{matrix*}[r]
y^0=&...&?&?&?&?&?&?&?&?&?&?&?&?&?&?&?&?&?&?&?&?&?&?&?&?&?&?&?&...\\y^1=&...&?&0&?&1&?&0&?&1&?&0&?&1&?&0&?&1&?&0&?&1&?&0&?&1&?&0&?&...\\
y^2=&...&0&0&?&1&1&0&?&1&0&0&?&1&1&0&?&1&0&0&?&1&1&0&?&1&0&0&?&...
\\
y^2=&...&0&0&0&1&1&0&?&1&0&0&1&1&1&0&?&1&0&0&0&1&1&0&?&1&0&0&1&...
\\
y^3=&...&0&0&0&1&1&0&0&1&0&0&1&1&1&0&?&1&0&0&0&1&1&0&1&1&0&0&1&...
\\\vdots
\\w=&...&0&0&0&1&1&0&0&1&0&0&1&1&1&0&0&1&0&0&0&1&1&0&1&1&0&0&1&...
\end{matrix*} 
\]
\end{example}

The Toeplitz subshift in Example \ref{ex1} has $(2^n)$ as a period structure and the one in Example \ref{ex2} has period structure $(4^n)$. Hence, both examples have $(2^n)$ as a prime scale. These two systems are not conjugate and neither is a factor of the other (see \cite{DownSur}). Our methods consist on finding the highest order of finite subgroups and how this number increases along different sequences of contentions of the form
$$\Aut(X,\sigma)\subseteq\Aut(X,\sigma^p)\subeq\Aut(X,\sigma^{p^2})\subseteq\Aut(X,\sigma^{p^3})\subseteq...$$
for all primes $p$. Because the Toeplitz subshift in Example \ref{ex2} has period structure $(4^n)$, when we consider $\Aut(X_{w},\sigma^{2^j})$ for $j$ odd, since $2^j$ divides a period, $\Aut(X_{w},\sigma^{2^j})$ contains a subgroup of order $2^j!$. Notice $\Aut(X_{u},\sigma^{2^j})$ also contains a subgroup of order $2^j!$ and in this case $2^j$ is an essential period. The methods we have developed so far, cannot distinguish between these two scenarios.

\bibliographystyle{amsplain}
\bibliography{sample}

\end{document}